\documentclass{article}
\usepackage[utf8]{inputenc}
\usepackage{amsthm,amsmath,tikz,amssymb,fullpage,enumerate,mathrsfs}

\newtheorem{theorem}{Theorem}[section]
\newtheorem{cor}[theorem]{Corollary}
\newtheorem{lemma}[theorem]{Lemma}
\newtheorem{prop}[theorem]{Proposition}

\newcommand{\pslr}{\mathrm{PSL}_2(\mathbb{R})}
\newcommand{\pslz}{\mathrm{PSL}_2(\mathbb{Z})}

\DeclareMathOperator{\arccosh}{arccosh}
\DeclareMathOperator{\tr}{tr}

\title{Hyperbolic punctured spheres without arithmetic systole maximizers}
\author{Grant S. Lakeland and Clayton Young }

\begin{document}

\maketitle

\begin{abstract}


We find bounds for the length of the systole -- the shortest essential, non-peripheral closed curve -- for arithmetic punctured spheres with $n$ cusps, for $n=4$ through $n=12$, some of which were previously known due to Schmutz. This is shown using a correspondence between such surfaces and planar triangulations. We show that for $n=7,10,11$, arithmetic surfaces do not achieve the maximal systole length. \end{abstract}

\section{Introduction}

Let $\Gamma$ be a discrete, cofinite, torsion-free subgroup of $\pslr$ with the property that the quotient surface $M=\mathbb{H}^2 / \Gamma$ has genus zero; i.e. that the surface is a punctured sphere without cone points. The \emph{systole} of $M$ is the shortest essential, non-peripheral loop on $M$. For a fixed topological type of $M$ -- which here means a fixed number of punctures $n$ -- there is much interest in the question of how long the length $\mathrm{sys}(M)$ of the systole may be among surfaces $M$ with the given topological type, and what surface(s) could achieve this maximum. 

One may suspect that a surface which achieves the maximum systole length could have a relatively large order automorphism group, given that such a surface likely has many distinct (isotopy classes of) geodesic loops which share the same length. Accordingly, one may ask whether surfaces maximizing the systole length must necessarily be arithmetic. Somewhat contrary to this heuristic, it has been shown by Fortier Bourque and Rafi \cite{FBR} that there exist hyperbolic surfaces which locally maximize the systole length and which have trivial automorphism group. In this paper, we show:

\begin{theorem}\label{thm:mainthm} For hyperbolic punctured spheres with, respectively, $n=7, 10$, and $11$ punctures, the systole length is not maximized by an arithmetic surface.\end{theorem}

This is shown by providing a bound on the systole length among arithmetic surfaces of each given topological type. A result of Schmutz \cite{Schmutz} (see also Fanoni--Parlier \cite{FanoniParlier}) gives an upper bound on $\mathrm{sys}(M)$ given its topological type of $M$; when the surface has genus zero and $n$ cusps, where necessarily $n \geq 4$, this bound becomes
\begin{equation}\label{bound} \mathrm{sys}(M) \leq 4 \arccosh{ \left( \frac{3n-6}{n} \right)}. \end{equation}
Schmutz showed that this bound is sharp in certain cases, including punctured spheres with, respectively, $n=4$, $n=6$ and $n=12$ punctures, as in these cases the systole is maximized by the surface corresponding to a principal congruence subgroup of $\pslz$ (see Section 2). 

In this paper, we study the cases $n=4$ through $n=12$ inclusive. We find that arithmetic punctured spheres have their systole lengths restricted by the graph theory of spherical (and hence planar) triangulations, where the vertices of the triangulation correspond to the cusps of the surface, and each face of the triangulation to a triangle in a tessellation of the upper half-plane. Specifically, the trace $\tr(\gamma)$ of the corresponding hyperbolic element $\gamma$ of $\Gamma$ is related to the length $\ell$ of the loop via
\begin{equation}\label{trlength} \ell = 2 \arccosh{\left( \frac{| \mathrm{tr}(\gamma) |} {2} \right)}. \end{equation}
We show that for each edge of the triangulation, the trace of the hyperbolic element corresponding to the loop encircling the two corresponding cusps is a function of the product of the two vertex degrees, and therefore an upper bound on the minimum of such products in a triangulation gives an upper bound on the systole length.

As a consequence of this, away from the cases of $n=4,6,12$, we obtain, for arithmetic examples, a sharper upper bound on the systole length than is given by inequality (\ref{bound}). This raises the question of whether there are non-arithmetic examples with longer systole, and we give such examples in the cases of $n=7,10,11$.

The paper is organized as follows. In Section 2, we show that the study of arithmetic punctured spheres is equivalent to the study of planar triangulations. In Section 3 we show explicitly how this correspondence works, and how one may obtain a (conjugacy class of) arithmetic Fuchsian group(s) from a planar triangulation equipped with a spanning tree subgraph. In Section 4, we show how one may bound the systole from the triangulation. In Section 5, we find bounds for arithmetic examples with $n=4$ through $n=12$ cusps, and construct examples with longer systoles in the cases of $n=7,10,11$. In Section 6 we outline some open questions which follow from this work.\\

\noindent {\bf Acknowledgements.} We thank Rick Anderson for helping initiate this project, Alan Reid for suggesting a shorter proof of Lemma 2.1 than was given in an earlier draft, and Brendan McKay for graph theory help with the $8$-cusped case via comments to a MathOverflow question \cite{McKay}.

\section{Arithmetic Fuchsian groups and triangulations}

In this section, we show that it suffices to consider triangulations of the sphere, and hence of the plane, when studying arithmetic hyperbolic spheres with cusps. 

\begin{lemma}If $\Gamma$ is a non-cocompact, torsion-free, arithmetic Fuchsian group of genus zero, then $\Gamma$ is (after possibly conjugating by an element of $\mathrm{PSL}_2(\mathbb{R})$) a subgroup of $\pslz$.\end{lemma}

\begin{proof}Let $\Gamma$ be such a group. By Takeuchi \cite{Takeuchi}, $\Gamma$ must be commensurable with $\pslz$. Hence its invariant trace field $k(\Gamma) = \mathbb{Q}((\mathrm{tr}\, \gamma)^2 : \gamma \in \Gamma)$ is $\mathbb{Q}$. Since $\Gamma$ is generated by parabolic elements, by Corollary 2.3 of Neumann--Reid \cite{NeumannReid}, the trace field $\mathbb{Q}(\mathrm{tr}\, \gamma : \gamma \in \Gamma)$ is equal to $k(\Gamma)$, and hence is also $\mathbb{Q}$. Since $\Gamma$ is assumed to be arithmetic, the traces of elements of $\Gamma$ are algebraic integers belonging to the trace field $\mathbb{Q}$, and so all traces of elements of $\Gamma$ belong to the rational integers $\mathbb{Z}$. By section 3.2 of Maclachlan--Reid \cite{ReidBook}, $\mathcal{O}(\Gamma)$ is an order in the quaternion algebra $A_0(\Gamma)$. But since $\Gamma$ is commensurable with $\pslz$, $A_0(\Gamma) = \mathrm{M}_2(\mathbb{Q})$ and so $\Gamma$ may be conjugated into $\mathrm{M}_2(\mathbb{Z})$. Thus $\Gamma$ can be seen as being contained in $\pslz$ as required. \end{proof}

%
\noindent {\bf The modular tessellation.} Recall that there is a tessellation (sometimes referred to as the Farey tessellation) which has one vertex for each element of $\mathbb{Q} \cup \{ \infty \}$, and an edge between $p/q$ and $r/s$ if and only if $ps-rq = \pm 1$, where we take fractions in lowest terms, and write $\infty = 1/0$. The sets of vertices, edges, and faces respectively are preserved, and acted upon, by the modular group $\pslz$. 

Each non-trivial parabolic element of $\pslz$ fixes one vertex of this tessellation and no edges or triangles. Only elliptic elements send edges or triangles to themselves. We may consider the quotient of the Farey tessellation by $\Gamma$; since $\Gamma$ is torsion-free, this will be a triangulation of the surface $M=\mathbb{H}^2 / \Gamma$ by ideal triangles. Moreover, since $\Gamma$ is cofinite, the quotient it will have finitely many edges and triangles, and hence finitely many (ideal) vertices. If we include the points $\mathbb{Q} \cup \{ \infty \}$ of the boundary of $\mathbb{H}^2$ then the quotient will be a finite graph which gives a triangulation of the sphere. By stereographic projection from any point in the interior of a triangle, this is equivalent to a planar graph where each face, including the exterior, is a triangle. We summarize this discussion in the below result.\\

\begin{cor}If $\Gamma$ is a non-cocompact, torsion-free, arithmetic Fuchsian group of genus zero, then the quotient surface $\mathbb{H}^2/\Gamma$ is triangulated by the quotient by $\Gamma$ of the modular tessellation.  \end{cor}

\noindent {\bf Principal congruence subgroups.} Given $N \in \mathbb{N}$, the \emph{principal congruence subgroup} of $\pslz$ of level $N$ is defined to be the kernel of the natural reduction homomorphism
\[ \varphi_N: \pslz \to \mathrm{PSL}_2(\mathbb{Z} / N\mathbb{Z}) \]
defined by reducing the entries modulo $N$. This group is denoted $\Gamma(N)$.

\section{Planar triangulations}

In this section, we show that conversely to the previous section, each triangulation of the plane corresponds to a conjugacy class of arithmetic, genus zero surface groups, and we give a procedure by which one may construct the group from a triangulation.

Suppose we have a graph $G$ which gives a planar triangulation. We construct a subgroup of $\pslz$ as follows. First, select a maximal connected spanning tree $\mathcal{T} \subset G$ and a terminal edge of $\mathcal{T}$, that is, an edge $e_{0,\infty}$ of $\mathcal{T}$ with a vertex $v_\infty$ of $\mathcal{T}$-degree 1. Next, label the other vertex of the edge $e_{0,\infty}$ as $v_0$. This edge is incident to two triangles; choose one of these triangles, label it $f_{0,1,\infty}$ and label the third vertex of the triangle as $v_1$. Label the other two edges $e_{1,\infty}$ and $e_{0,1}$ according to the labels of their respective vertices. 

Given this face $f_{0,1,\infty}$ as a starting point, we proceed by defining a process of developing the labeling across an edge of $G \setminus \mathcal{T}$. Either or both of the edges $e_{0,1}$ and $e_{1,\infty}$ do not belong to $\mathcal{T}$; if both, choose one which does not belong to $\mathcal{T}$. Label the face adjacent to $f_{0,1,\infty}$ across this edge either $f_{1,2,\infty}$ or $f_{0,1/2,1}$ depending on which face is adjacent across the corresponding edge in the modular tessellation. Label the third vertex of that face as $v_2$ or $v_{1/2}$ and label the two edges according to their vertices. Proceed in the same manner, at each step selecting a labeled edge of $G \setminus \mathcal{T}$ adjacent to an unlabeled face. Each time a face is labeled, label the edges and vertex according to the modular tessellation, and add duplicate labels if they already have labels. Continue until every face has a label, when each vertex will have the same number of labels as its $\mathcal{T}$-degree, and each edge of $\mathcal{T}$ has two labels. 

Once this process is complete, pass to the upper half-plane, and construct a polygon $P$ which consists of all the triangles of the Farey tessellation which correspond to the labels on the faces. The boundary of the union of these triangles consists of the edges of $\mathcal{T}$. For each such edge, the two edges corresponding to its labels will be identified by side-pairings of $P$. For example, the edge $e_{0,\infty}$ will also receive a label $e_{d,\infty}$ where $d$ is the $G$-degree of $v_\infty$, and the two edges $e_{0,\infty}$ and $e_{d,\infty}$ of $P$ will be identified by $T^d$. An edge of $\mathcal{T}$ which is not terminal and has labels of $v_{p_1/q_1}$ and $v_{r_1/s_1}$ on one side, and $v_{p_2/q_2}$ and $v_{r_2/s_2}$ on the other, where the letters match on vertices, gives rise to the element of $\pslz$ which sends $p_1/q_1$ to $p_2/q_2$ and $r_1/s_1$ to $r_2/s_2$.

For each vertex $v_{p/q}$ of $G$, it will be a consequence of this construction that our group $\Gamma$ contains the element which is the conjugate of $T^d$, where $d$ is the degree of $v_{p/q}$, which fixes $p/q$. \\

\begin{lemma}\label{lem:fixed}For each vertex label $v_{p/q}$ in the above construction, if $d$ denotes the degree of the vertex of $G$ at that label, the group $\Gamma$ contains the $\pslz$-conjugate of $T^d$ which fixes $p/q$.\end{lemma}

\begin{proof} If the vertex in question has $\mathcal{T}$-degree 1, then the result follows from the fact that the corresponding element of $\Gamma$ fixes $p/q$ and translates its neighbors $d$ triangles over. Thus, suppose that the vertex does not have $\mathcal{T}$-degree 1, and let $m$ denote this $\mathcal{T}$-degree. By conjugation, we may move the vertex to $\infty$. Then there are $m$ elements $\gamma_1, \ldots, \gamma_m \in \Gamma$, corresponding to the $m$ edges incident to the vertex, so that the labels at the vertex are $\infty$, $\gamma_1(\infty), \ldots, \gamma_{m-1}(\infty)$, and such that $\gamma_m \circ \gamma_{m-1} \circ \ldots \circ \gamma_2 \circ \gamma_1(\infty) = \infty$. Let $\gamma = \gamma_m \circ \gamma_{m-1} \circ \ldots \circ \gamma_2 \circ \gamma_1$. Then $\gamma \in \Gamma$ and hence is an element of $\pslz$; further, $\gamma$ fixes $\infty$ and thus must be a parabolic element. It follows that $\gamma$ is a power of $T$, and it remains to show that $\gamma = T^d$. To see this, observe that for each triangle incident to the vertex, one of the elements $\gamma_1^{-1}, \gamma_1^{-1} \circ \gamma_2^{-1}, \ldots, \gamma_1^{-1} \circ \gamma_2^{-1} \circ \ldots \circ \gamma_{m-1}^{-1}$ will send that triangle to one incident to the vertex at $\infty$; once we do this, we see $d$ triangles arranged in sequence, and the extreme edges of these are equivalent under the action of $\Gamma$. Therefore we see that $T^d$ identifies these faces, and note that this is true for no natural number less than $d$. After undoing the conjugation that moved our vertex to $\infty$, we see that $\Gamma$ contains the appropriate $\pslz$-conjugate of $T^d$.  \end{proof}

\noindent {\bf Example 1.} Let $G$ be a tetrahedron, as shown in Figure \ref{fig:tetra} with a spanning tree $\mathcal{T}$.

\begin{figure}[htb]
\begin{center}
\begin{tikzpicture}
\filldraw (0,0) circle (3pt) -- (0,2) circle (3pt) -- (1.7,-1) circle (3pt) -- (-1.7,-1) circle (3pt);
\filldraw[very thick] (-1.7,-1) -- (0,2);
\filldraw[very thick] (-1.7,-1) -- (0,0);
\filldraw[very thick] (-1.7,-1) -- (1.7,-1);
\filldraw (1.7,-1) -- (0,0);
\end{tikzpicture}
\caption{The thicker lines are the spanning tree $\mathcal{T}$}
\label{fig:tetra}
\end{center}
\end{figure}
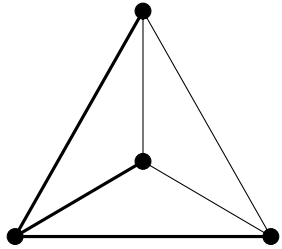

\noindent  Select the central vertex as $v_{\infty}$ and its neighbor in $\mathcal{T}$ as $v_0$. Select the lower of the three triangles as $f_{0,1,\infty}$. This makes the lower right vertex $v_1$ and the two other edges of that triangle $e_{0,1}$ and $e_{1,\infty}$. At this point, we have only one available edge of $G \setminus \mathcal{T}$, it is the edge $e_{1,\infty}$. We label the top right triangle $f_{1,2,\infty}$, the top vertex $v_2$, and the edges $e_{1,2}$ and $e_{2,\infty}$. If we now choose $e_{2,\infty}$ then the left triangle is $f_{2,3,\infty}$, and then if we develop across $e_{1,2}$ then the back face of the tetrahedron is $f_{1,3/2,2}$. The complete labeled graph is shown in Figure \ref{fig:labeled}.

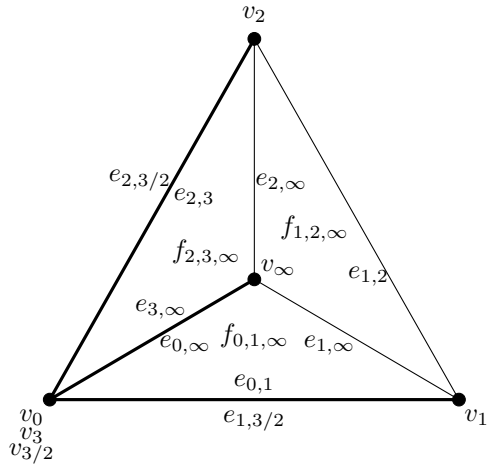
\begin{figure}[htb]
\begin{center}
\begin{tikzpicture}[scale=0.8]
\filldraw (0,0) circle (3pt) -- (0,4) circle (3pt) -- (3.4,-2) circle (3pt) -- (-3.4,-2) circle (3pt);
\filldraw[very thick] (-3.4,-2) -- (0,4);
\filldraw[very thick] (-3.4,-2) -- (0,0);
\filldraw[very thick] (-3.4,-2) -- (3.4,-2);
\filldraw (3.4,-2) -- (0,0);
\draw (0.4,0.2) node {$v_\infty$};
\draw (-3.7,-2.3) node {$v_0$};
\draw (-3.7,-2.6) node {$v_3$};
\draw (-3.7,-2.9) node {$v_{3/2}$};
\draw (3.7,-2.3) node {$v_1$};
\draw (0,4.4) node {$v_2$};
\draw (-1.15,-1.1) node {$e_{0,\infty}$};
\draw (0,-1.75) node {$e_{0,1}$};
\draw (0,-2.35) node {$e_{1,3/2}$};
\draw (1.25,-1.1) node {$e_{1,\infty}$};
\draw (0.45,1.6) node {$e_{2,\infty}$};
\draw (1.9,0.05) node {$e_{1,2}$};
\draw (-1.55,-0.5) node {$e_{3,\infty}$};
\draw (0,-0.95) node {$f_{0,1,\infty}$};
\draw (1,0.85) node {$f_{1,2,\infty}$};
\draw (-0.8,0.45) node {$f_{2,3,\infty}$};
\draw (-1,1.35) node {$e_{2,3}$};
\draw (-1.9,1.65) node {$e_{2,3/2}$};
\end{tikzpicture}
\caption{The graph after labeling is complete. The back face is $f_{1,3/2,2}$}
\label{fig:labeled}
\end{center}
\end{figure}

When we draw the corresponding triangles in the upper half-plane we see Figure \ref{fig:farey}.

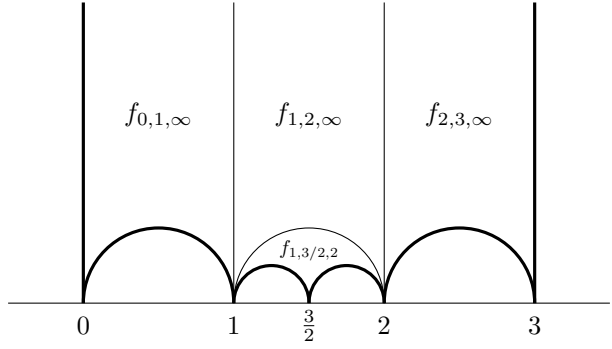
\begin{figure}[htb]
\begin{center}
\begin{tikzpicture}
\filldraw (-1,0) -- (7,0);
\filldraw (2,0) -- (2,4);
\filldraw (4,0) -- (4,4);
\draw (4,0) arc (0:180:1);

\filldraw[very thick] (0,0) -- (0,4);
\filldraw[very thick] (6,0) -- (6,4);
\draw[very thick] (2,0) arc (0:180:1);
\draw[very thick] (6,0) arc (0:180:1);
\draw[very thick] (3,0) arc (0:180:0.5);
\draw[very thick] (4,0) arc (0:180:0.5);

\draw (0,-0.3) node {$0$};
\draw (2,-0.3) node {$1$};
\draw (4,-0.3) node {$2$};
\draw (6,-0.3) node {$3$};
\draw (3,-0.3) node {$\frac{3}{2}$};
\draw (1,2.5) node {$f_{0,1,\infty}$};
\draw (3,2.5) node {$f_{1,2,\infty}$};
\draw (5,2.5) node {$f_{2,3,\infty}$};
\draw (3,0.7) node[scale=0.75] {$f_{1,3/2,2}$};

\end{tikzpicture}
\caption{The thicker lines from $\mathcal{T}$ become the boundary of the fundamental domain}
\label{fig:farey}
\end{center}
\end{figure}

We generate the group $\Gamma$ by taking $T^3$, since $v_\infty$ has degree 3. We take the $\pslz$-conjugates of $T^3$ which fix 1 and 2 respectively; note that these pair the sides $e_{0,1}$ with $e_{1,3/2}$ and $e_{3/2,2}$ with $e_{2,3}$ respectively as each moves the tiling 3 triangles. These suffice to generate $\Gamma$ since they pair all the sides of $P$. We note that we could include the conjugates of $T^3$ fixing $0$, $3/2$, and $3$, but these are conjugate to one another via the other generators, and also they can be written as products of the other generators.\\

\noindent {\bf Example 2.} The triangulation shown in Figure \ref{fig:casen10} has 10 vertices; the labeling described above is given for the vertices only as the edge and face information may be inferred from this. We may read off the vertices of the fundamental polygon that arises from this diagram by starting at $\infty$ and moving along $\mathcal{T}$; it has vertices at $\infty, 0, 1/2, 1, 3/2, 2, 7/3, 5/2, 3, 10/3, 7/2, 18/5, 29/8, 11/3, 4, 9/2, 14/3, 5, \infty$.

\begin{figure}[htb]
\begin{center}
\begin{tikzpicture}[scale=0.9]
\draw (0,-2) -- (0,2) -- (2,0) -- (-2,0) -- (0,-2);
\draw (-2,0) -- (0,2);
\draw (0,-2) -- (2,0);
\draw (3,3) -- (3,-3) -- (-3,-3) -- (-3,3) -- (3,3);
\draw (3,3) -- (2,0) -- (3,-3) -- (0,-2) -- (-3,-3) -- (-2,0) -- (-3,3) -- (0,2) -- (3,3);
\draw (-3,3) -- (0,6) -- (3,3);

\draw (0,6) .. controls (-8,3) and (-4,-2) .. (-3,-3);
\draw (0,6) .. controls (8,3) and (4,-2) .. (3,-3);

\draw[ultra thick] (0,0) -- (2,0) -- (0,2);
\draw[ultra thick] (2,0) -- (0,-2) -- (-2,0);
\draw[ultra thick] (3,3) -- (2,0) -- (3,-3);
\draw[ultra thick] (-3,-3) -- (3,-3);
\draw[ultra thick] (-3,3) -- (3,3) -- (0,6);

\draw (-2.3,0) node {$v_\infty$};
\draw (-0.2,-1.5) node {$v_0$};
\draw (-0.2,0.2) node {$v_1$};
\draw (0,2.2) node {$v_2$};
\draw (-3.2,3.1) node {$v_3$};
\draw (-3.1,-3.2) node {$v_4$};
\draw (1.4,-0.2) node {$v_{1/2}$};
\draw (1.4,0.2) node {$v_{3/2}$};
\draw (1.7,2.8) node {$v_{5/2}$};
\draw (1.85,0.6) node {$v_{7/3}$};
\draw (0,6.2) node {$v_{7/2}$};
\draw (-0.35,-1.9) node {$v_{5}$};
\draw (2.5,-2.6) node {$v_{9/2}$};
\draw (1.8,-0.8) node {$v_{14/3}$};
\draw (2.2,3.2) node {$v_{10/3}$};
\draw (3.2,-3.25) node {$v_{11/3}$};
\draw (3.35,3.2) node {$v_{18/5}$};
\draw (2.45,0) node {$v_{29/8}$};

\end{tikzpicture}
\caption{One triangulation for $n=10$ with a spanning tree $\mathcal{T}$ in bolded lines}
\label{fig:casen10}
\end{center}
\end{figure}
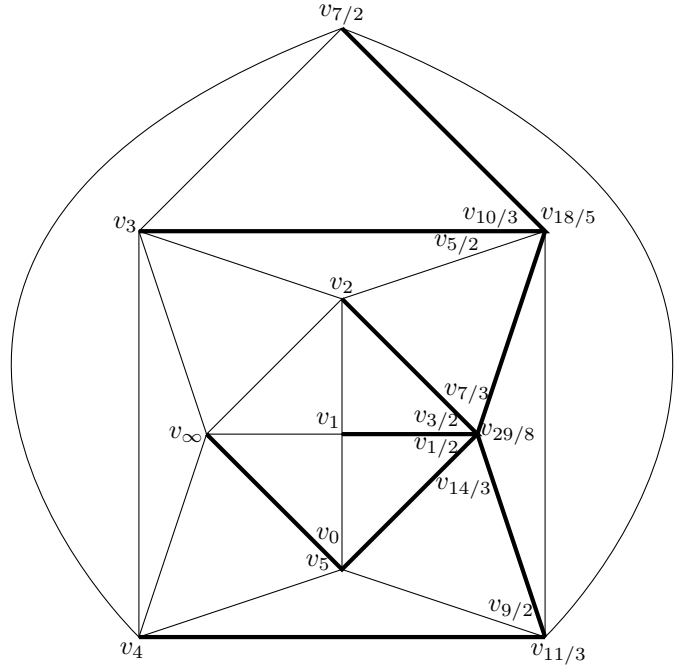

The generators of the group $\Gamma$ are then: $T^5$; the $\pslz$-conjugates of $T^5$ which fix $2, 3$, and $4$; the $\pslz$-conjugates of $T^4$ which fix $1$ and $7/2$; and the elements
\[ \begin{pmatrix} -24 & 5 \\ -5 & 1\end{pmatrix}, \begin{pmatrix} -114 & 415 \\ -25 & 91\end{pmatrix}, \mbox{ and } \begin{pmatrix} -112 & 271 \\ -31 & 75\end{pmatrix} \]
which pair, respectively, $e_{0,1/2}$ with $e_{5,14/3}$, $e_{29/8,11/3}$ with $e_{14/3,9/2}$, and $e_{7/3,5/2}$ with $e_{18/5,29/8}$.\\

\noindent {\bf Definition.} We denote by $\delta(G)$ the minimum degree of a vertex in the graph $G$. \\

Observe that in Example 1, $\delta(G) = 3$, and in Example 2, $\delta(G)=4$. We will for the most part consider triangulations with $\delta(G) \geq 3$, but we will also need to consider graphs with vertices of degree 1 or 2.\\

\noindent {\bf Example 3.} In Figure \ref{fig:deg2ex} we have a triangulation with five vertices, including one of degree 2. We see two edges which have the same pair of vertices; this is permitted in our discussion and both edges count towards the degrees of the respective vertices.\\

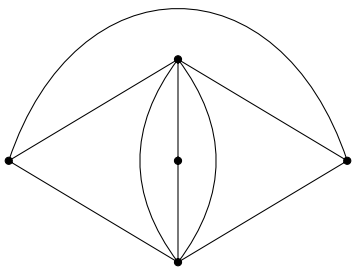
\begin{figure}[htb]
\begin{center}
\begin{tikzpicture}[scale=0.45]
\filldraw (0,-3) circle (3pt) -- (0,0) circle (3pt)-- (0,3) circle (3pt);
\filldraw (5,0) circle (3pt) -- (0,-3);
\filldraw (-5,0) circle (3pt) -- (0,-3);
\draw (5,0) -- (0,3) -- (-5,0);

\draw (0,-3) .. controls (1.5,-1) and (1.5,1) .. (0,3);
\draw (0,-3) .. controls (-1.5,-1) and (-1.5,1) .. (0,3);

\draw (-5,0) .. controls (-3,6) and (3,6) .. (5,0);

\end{tikzpicture}
\caption{A triangulation with six triangles which includes a vertex of degree 2}
\label{fig:deg2ex}
\end{center}
\end{figure}

\noindent {\bf Example 4.} In Figure \ref{fig:deg1ex} we have a triangulation with four vertices, one of which has degree 1. This corresponds to a genus zero subgroup of $\pslz$ which contains $T$ (or some $\pslz$-conjugate thereof) and where therefore two of the sides of a modular triangle are identified, creating the central circular part of the figure. Note that this creates an edge which does not have two distinct vertices, but the same vertex twice. Both occurrences will be counted in the degree of this vertex; for example, Figure \ref{fig:deg1ex} has a vertex of degree 6.\\

\begin{figure}[htb]
\begin{center}
\begin{tikzpicture}[scale=0.5]
\draw (0,0) circle (2);

\filldraw (2,0) circle (3pt);
\filldraw (-4,0) circle (3pt);
\filldraw (0,0) circle (3pt) -- (2,0);
\filldraw (4,0) circle (3pt) -- (2,0);

\draw (-4,0) .. controls (-3.7,5) and (1.7,5) .. (2,0);
\draw (-4,0) .. controls (-3.7,-5) and (1.7,-5) .. (2,0);
\draw (-4,0) .. controls (-3.5,6) and (3.5,6) .. (4,0);

\end{tikzpicture}
\caption{A triangulation with four triangles which includes a vertex of degree 1}
\label{fig:deg1ex}
\end{center}
\end{figure}
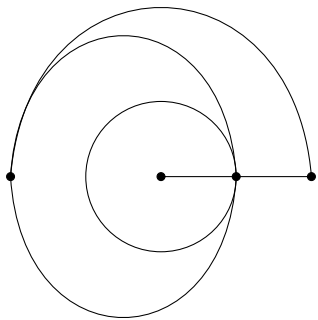

\noindent {\bf Definitions.} In the remainder of this paper, a triangulation $G$ will be called \emph{regular} if $\delta(G) \geq 3$, each edge has two distinct endpoints, and no two edges have the same two endpoints. An edge which does not have two distinct endpoints will be called a \emph{loop}, and distinct edges with the same endpoints will be called \emph{duplicate edges}. If a vertex $v$ has degree 1, we will refer to the sole triangle $\mathscr{T}$ incident to $v$ as \emph{containing} $v$, and the unique triangle sharing an edge with $\mathscr{T}$ as the triangle \emph{enclosing} $\mathscr{T}$.\\

Note that Example 3 (see Figure \ref{fig:deg2ex}) has duplicate edges, and Example 4 (see Figure \ref{fig:deg1ex}) has both a loop and a pair of duplicate edges.

\section{Systoles from triangulations}

In this section, we describe how one can take a spherical triangulation and deduce bounds for the systole, assuming that the triangulation corresponds to a subset of the modular triangulation as described above.

We will make use of the following Lemma in this section and in later sections.

\begin{lemma}\label{tracemn}If $\Gamma$ contains parabolic elements 
\[ \begin{pmatrix} 1 & m_1 \\ 0 & 1 \end{pmatrix} \ \ \mbox{    and     } \ \ \begin{pmatrix} a & b \\ m_2 & d \end{pmatrix}\]
where $a+d=2$, then $\Gamma$ contains an element $\gamma$ with $|\tr{(\gamma)}| = |m_1 m_2 -2|$.  \end{lemma}

\begin{proof} We have
\[ \begin{pmatrix}1 & m_1 \\ 0 & 1 \end{pmatrix}   \begin{pmatrix}a & b \\ m_2 & d \end{pmatrix}^{-1} =  \begin{pmatrix}1 & m_1 \\ 0 & 1 \end{pmatrix}   \begin{pmatrix}d & -b \\ -m_2 & a \end{pmatrix} = \begin{pmatrix} d-m_1m_2 & am_2-b \\ -m_1 & a \end{pmatrix}. \]
Since $a+d=2$, the trace of this element is $2-m_1m_2$, giving the stated result.  \end{proof}

%

We may also see the above result in terms of the L and R matrices (see Series \cite{Series}), corresponding to turning left or right through the triangulation of the surface, and the path as going around two faces of the dual trivalent graph to $G$, making either a left or right turn at each vertex (see Brooks--Makover \cite{BrooksMakover}). A loop which goes around two cusps of the triangulation can be seen as one left turn, $m_1-2$ right, one left, and $m_2-2$ right, where $m_1$ and $m_2$ are the degrees of the vertices (see Figure \ref{fig:lsandrs}). With the associations that 
\[ L = \begin{pmatrix}1 & 1 \\ 0 & 1\end{pmatrix} \ \ \ \ \ R = \begin{pmatrix}1 & 0 \\ 1 & 1\end{pmatrix},\]
we find that 
\begin{align*} LR^{m_1-2}LR^{m_2-2} &= \begin{pmatrix}1 & 1 \\ 0 & 1\end{pmatrix}  \begin{pmatrix}1 & 0 \\ m_1-2 & 1\end{pmatrix}  \begin{pmatrix}1 & 1 \\ 0 & 1\end{pmatrix}  \begin{pmatrix}1 & 0 \\ m_2-2 & 1\end{pmatrix} \\
&= \begin{pmatrix}m_1-1 & 1 \\ m_2-2 & 1\end{pmatrix} \begin{pmatrix}m_2-1 & 1 \\ m_2-2 & 1\end{pmatrix} \\
&= \begin{pmatrix} (m_1-1)(m_2-1) +m_2-2 & m_1 \\ (m_1-2)(m_2-1)+m_2-2 & m_1-1     \end{pmatrix} \end{align*}
which has trace $m_1m_2-m_1-m_2+1+m_2-2+m_1-1 = m_1m_2-2$. \\

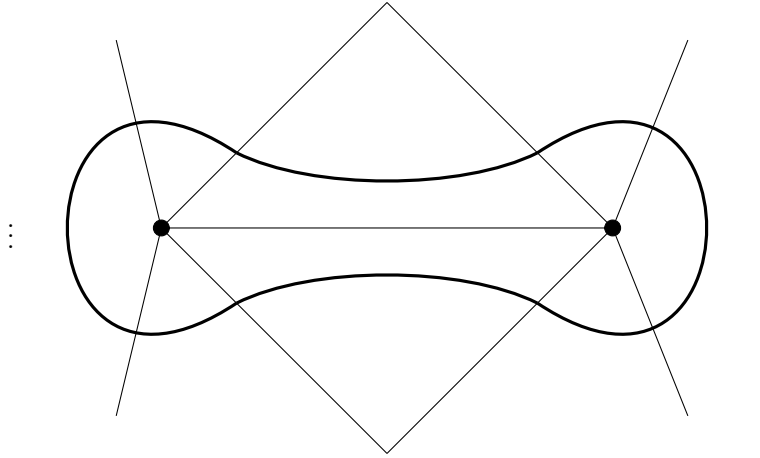
\begin{figure}[htb]
\begin{center}
\begin{tikzpicture}
\filldraw (-3,0) circle (3pt) -- (3,0) circle (3pt);
\draw (-3,0) -- (0,3);
\draw (-3,0) -- (0,-3);
\draw (-3,0) -- (-3.6,2.5);
\draw (-3,0) -- (-3.6,-2.5);
\draw (3,0) -- (0,3);
\draw (3,0) -- (0,-3);
\draw (3,0) -- (4,2.5);
\draw (3,0) -- (4,-2.5);

\draw (-5,0) node {$\vdots$};
\draw (5,0) node {$\vdots$};

\draw[very thick] (-2,-1) .. controls (-5,-3) and (-5,3) .. (-2,1);
\draw[very thick] (2,-1) .. controls (5,-3) and (5,3) .. (2,1);
\draw[very thick] (-2,1) .. controls (-1,0.5) and (1,0.5) .. (2,1);
\draw[very thick] (-2,-1) .. controls (-1,-0.5) and (1,-0.5) .. (2,-1);

\end{tikzpicture}
\caption{If the vertices have degrees $m_1$ and $m_2$ then the loop corresponds to the word $LR^{m_1-2}LR^{m_2-2}$}
\label{fig:lsandrs}
\end{center}
\end{figure}

\noindent {\bf Definition: Density of an edge.} We henceforth define the \emph{density} $D(e)$ of an edge $e$ in a triangulation to be the product of the two degrees of its vertices; that is, if the two vertices have degrees $m_1$ and $m_2$ respectively, then $D(e) = m_1m_2$. \\

In the next section we obtain the bounds on the systole for arithmetic examples through bounding the possible densities of edges in a planar graph. We will make use of the following result, which is a reframing of Lemma \ref{tracemn} with the density in mind.

\begin{lemma}\label{densitylength}An edge of density $D$ in the graph $G$ corresponds to a loop which goes around the corresponding cusps and has length $2\arccosh{\left( (D-2)/2) \right)}$.  \end{lemma}

\begin{proof}This is a combination of Lemma \ref{tracemn} and equation (\ref{trlength}).  \end{proof}
\begin{lemma}\label{adjtriangles}If there are two vertices of adjacent triangles whose degrees are 2 and, respectively, 2 or 3, then there is a hyperbolic element in the corresponding group of trace, respectively, 14 or 22.  \end{lemma}

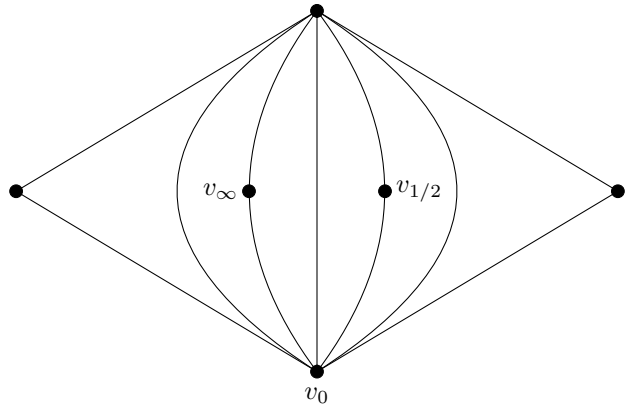
\begin{figure}[htb]
\begin{center}
\begin{tikzpicture}[scale=0.8]
\filldraw (0,-3) circle (3pt) -- (0,3) circle (3pt);
\filldraw (5,0) circle (3pt) -- (0,-3);
\filldraw (-5,0) circle (3pt) -- (0,-3);
\draw (5,0) -- (0,3) -- (-5,0);

\draw (0,-3) .. controls (1.5,-1) and (1.5,1) .. (0,3);
\draw (0,-3) .. controls (-1.5,-1) and (-1.5,1) .. (0,3);
\draw (0,-3) .. controls (3.1,-1) and (3.1,1) .. (0,3);
\draw (0,-3) .. controls (-3.1,-1) and (-3.1,1) .. (0,3);

\filldraw (1.13,0) circle (3pt);
\filldraw (-1.13,0) circle (3pt);

\draw (-1.6,0) node {$v_\infty$};
\draw (0,-3.4) node {$v_0$};
\draw (1.7,0) node {$v_{1/2}$};

\end{tikzpicture}
\caption{The two degree 2 vertices are vertices of adjacent triangles}
\label{fig:degreetwos}
\end{center}
\end{figure}

\begin{proof}Select a spanning tree $\mathcal{T}$ which has a terminal vertex, i.e. one of $\mathcal{T}$-degree 1, at the degree 2 vertex, or one of the two degree 2 vertices. Further, select $\mathcal{T}$ so that it does not contain the edge common to the two adjacent triangles. Assign $v_\infty$ to be the terminal vertex of $\mathcal{T}$ and $v_0$ to be the other vertex of the edge of $\mathcal{T}$. Then the other vertex of degree 2 or 3 will be assigned to be $v_{1/2}$ (see Figure \ref{fig:degreetwos}).

With this setup, the group $\Gamma$ will contain 
\[ T^2 = \begin{pmatrix} 1 & 2 \\ 0 & 1 \end{pmatrix} \]
and the conjugates of $T^2$ or $T^3$ fixing $1/2$; these are
\[ \begin{pmatrix} 1 & 0 \\ 2 & 1\end{pmatrix} \begin{pmatrix}1 & d \\ 0 & 1\end{pmatrix} \begin{pmatrix} 1 & 0 \\ -2 & 1\end{pmatrix} = \begin{pmatrix}1-2d & d \\ -4d & 2d+1\end{pmatrix}    \]   
where $d$ is the degree of $v_{1/2}$. By Lemma \ref{tracemn} with $m_1=2$ and $m_2 = 8$ or $12$ depending on whether $d$ is equal to 2 or 3, we see that $\Gamma$ contains a hyperbolic element of trace 14 or 22.\end{proof}

\begin{lemma}\label{lem:loops}If $G$ contains a loop with both of its vertices being $v$, then either $G$ has a vertex of degree 1 adjacent to $v$, or $\Gamma$ has an element of trace no greater than $\left\lfloor \frac{d}{2} \right\rfloor$ , where $d=\mathrm{deg}(v)$. \end{lemma}

\begin{proof}If $G$ contains a loop $e$ based at $v$, then there is a triangle $\mathscr{T}$ of $G$ with at least two of its three vertices at $v$. It is possible that the third vertex of $\mathscr{T}$ has degree 1 and is adjacent to $v$. If this is the case, then the loop $e$ divides the sphere into two, but one of the components has only one triangle. Suppose now that $e$ divides the sphere into two components, each with at least two triangles. Then the edges incident to $v$ are divided between these two components, and one component must contain no more than $(d-2)/2$ of these edges. On this side, there will be a path through the dual trivalent graph which begins and ends in $\mathscr{T}$ of the form $L^{p} R$, where $p \leq d-3$, and hence an element of trace no greater than $d$.  \end{proof}

\begin{lemma}\label{lem:duplicates}If the triangulation $G$ has duplicate edges, then the duplicate edges separate the sphere into two components, each of which has an even number of triangles.\end{lemma}

\begin{proof}Choose either of the two components into which the duplicate edges separate the sphere, and delete all edges and vertices in the other component. This leaves a triangulation of a disk with two outside edges, which can be seen as a bigon. Removing one of these edges $e$ creates another triangulation of the sphere, which necessarily has an even number of triangles. All but one of these belonged to the component, and adding back in the last deleted edge $e$ creates one more triangle in the given component. So the component has an even number of triangles. \end{proof}

\begin{lemma}\label{lem:deg1s}Suppose $G$ has a degree 1 vertex $v_1$, adjacent to the vertex $v_2$, and suppose that the triangle which shares the looped side has a third vertex $v_3$ of degree $d$. Then $\Gamma$ has a hyperbolic element of trace $4d-2$. \end{lemma}

\begin{proof}There is a loop with word $L^4R^{d-1}$ which gives rises to an element of trace $4d-2$.  \end{proof}

\section{Maximizing systoles}

In this section, we consider surfaces with $n$ cusps for some small values of $n$. We find the arithmetic surface with $n$ cusps which maximizes the systole, by using the graph theory from the previous section. We make use of the equation from the following result:

\begin{lemma}\label{degreecount} For any planar triangulation, we have
\begin{equation}\label{grapheqn} \sum_{v \in V} (6 - \mathrm{deg}(v) ) = 12. \end{equation}\end{lemma}

\begin{proof}Recall Euler's formula, $|V|-|E|+|F|=2$, where $|V|$, $|E|$ and $|F|$ are the numbers of vertices, edges and faces respectively. We note that 
\[ |V| = \sum_{v \in V} 1\]
and
\[ 2|E| = \sum_{v\in V} \mathrm{deg}(v), \]
the latter equation because adding the degrees of all vertices counts every edge twice, once at each endpoint. Furthermore, we note that 
\[ 2|E| = 3|F| \]
because we have a triangulation, so the same count of all vertex degrees can be seen as counting each triangle three times. We therefore see that 
\[ 6|V| - 6|E| + 6|F| =12 \]
and so 
\[ 6\sum_{ v \in V} 1 - 3\sum_{v\in V} \mathrm{deg}(v) + 2\sum_{v \in V} \mathrm{deg}(v) = 12 \]
and therefore combining terms gives the stated equality.  \end{proof}

Note that equation (\ref{grapheqn}) holds even in the presence of vertices of degree 1 or 2, and we have chosen our conventions with this in mind. For example, Example 3 has five vertices of degrees 2, 5, 5, 3 and 3 respectively. Example 4 has vertices of degrees 1, 2, 3 and 6.

\subsection{$n=4$} Here there is only one possible regular triangulation, with $\delta(G) \geq 3$. It has four vertices of degree 3 and these form a tetrahedron, as in Example 1. Here every edge has density 9 and so the systole has length $2 \arccosh{(7/2)}$, which matches the bound in inequality (\ref{bound}). The group here is the principal congruence subgroup $\Gamma(3)$. \\

Although the bound from (\ref{bound}) being achieved means we need not consider graphs with $\delta(G) <3$, we may see in the light of equation (\ref{grapheqn}) that no such graph can give a longer systole. Suppose $\delta(G) =2$. Then the degree 2 vertex must be adjacent to vertices of degree at least 5 to achieve a longer systole via a larger minimum edge density. But then, if the two degree 5 vertices were distinct, the fourth vertex would need to have degree 0 to satisfy equation (\ref{grapheqn}). If the degree 5 vertices were the same vertex, then there must exist two loops, as the third edges of the two triangles incident to the degree 2 vertex. But then the vertices are necessarily 1, 1, 2 and 8 respectively, and there is an edge of density 8. 

If $\delta(G)=1$, a degree 1 vertex would need to have adjacent to it a vertex of degree at least 10 to create a longer systole. But this constrains the remaining two vertices to have degrees summing to 1.

\subsection{$n=5$} Here, if $\delta(G)\geq 3$, we will have three degree 4 vertices and two degree 3 vertices. We may avoid a density 9 edge by keeping the degree 3 vertices non-adjacent. The result has all edges of density either 12 or 16. We note here that the bound given by (\ref{bound}) is $4\arccosh{(9/5)} \approx 4.77164...$ and that thus a density 12 edge provides a shorter systole. 

To see that this is the longest systole among arithmetic examples, suppose there is a degree 2 vertex. To achieve a larger minimum edge density, it must be adjacent to vertices of degree at least 7. But then by equation (\ref{grapheqn}), if these are distinct, they must both be degree 7, and then the two remaining vertices must both be degree 1. The degree 1 vertices will then contribute to an edge of density less than 12. If the degree 2 vertex is incident to duplicate edges, and hence is adjacent to only one vertex of degree at least 7, then $G$ has a loop, and Lemma \ref{lem:loops} gives a smaller trace hyperbolic element.

Similarly, if we assume there is a degree 1 vertex, then it must be adjacent to a vertex of degree at least 13 to ensure a longer systole. By equation (\ref{grapheqn}), there must be at least three vertices of degree 1, and then the remaining two vertices have degrees 1 and 14, or 2 and 13. In both cases there must exist a loop, so by Lemma \ref{lem:loops} there must be a hyperbolic element of trace 6 or 7, shorter than the systole bound given above. \\

We note that we have not been able to find a non-arithmetic example with 5 cusps with a longer systole than the arithmetic example given here.

%
%
\subsection{$n=6$} If $\delta(G) \geq 3$ and one of the six vertices has degree 5, then it is adjacent to all other 5, one of which has degree 3. This means there must be an edge of density 15. If instead we take the octahedron where each vertex has degree 4, all edges have density 16. This corresponds to the principal congruence subgroup $\Gamma(4)$ and we note that this achieves equality in (\ref{bound}).\\
\subsection{$n=7$} Every arithmetic sphere with 7 cusps has systole of length at most $2 \arccosh{(7)}$.

\begin{prop}Every planar triangulation $G$ with seven vertices has an edge with density at most $16$ or vertices arranged as in Lemma \ref{adjtriangles}.\end{prop}

\begin{proof}Suppose $G$ is regular, has a degree 3 vertex, and $\delta(G)=3$. If it is adjacent to a vertex of degree 3, 4 or 5, we are done. Thus, suppose its three neighbors are all degree 6 vertices. By equation \ref{grapheqn}, the remaining vertices must all have degree 3. Then either two such are adjacent, meaning there is an edge of density 9, or each of the degree 3 vertices is adjacent to each of the degree 6 vertices. In this latter case, the graph contains a $K_{3,3}$ subgraph, and hence cannot be planar. Thus, if $G$ has a degree 3 vertex, it has an edge of density at most 15.

If $G$ does not have a degree 3 vertex and $\delta(G) \geq 4$, then, in light of equation \ref{grapheqn}, it must have either six degree 4 vertices, and one degree 6 vertex, or five degree 4 and two degree 5 vertices. In both cases, there must be two degree 4 vertices adjacent to one another, and so an edge of density 16 (see Figure \ref{fig:casen7}).

If $\delta(G) \geq 3$ but $G$ has duplicate edges, then since $G$ has ten triangles, by Lemma \ref{lem:duplicates}, one of the components has 2 or 4 triangles. If 2 triangles, $\delta(G) =2$, a contradiction, and if 4, then either there are two adjacent vertices, each of degree 3, or Lemma \ref{adjtriangles} applies and gives an element of trace at most 14.

If $\delta(G) =2$, then suppose $G$ has a degree 2 vertex adjacent to two vertices of degrees at least 9. The remaining four vertices have degrees which sum to at most 10, so there must be at least two of these of degree 2. But these must be adjacent to the two vertices of degree 9, meaning that Lemma \ref{adjtriangles} applies and gives an element of trace 14. If the degree 2 vertex has both edges incident to one vertex of degree at least 9, then $G$ has a loop. Since $\delta(G)=2$ the largest the degree of a vertex can be is 18. By Lemma \ref{lem:loops}, it then follows that there must be a hyperbolic element of trace at most 9.

If $\delta(G) = 1$, then suppose that each degree 1 vertex is adjacent to a vertex of degree at least 17. This degree implies that there is a loop not enclosing a degree 1 vertex, and hence there is a hyperbolic element of trace at most $9$.  \end{proof}
%
%
\begin{figure}[htb]
\begin{center}
\begin{tikzpicture}
\draw (-4,0.15) -- (-2.5,-1) -- (2.5,-1) -- (4,0.15);
\draw (0,4) -- (-4,0.15)  -- (0,-4);
\draw (0,4) -- (4,0.15)  -- (0,-4);
\draw (0,4) -- (-2.5,-1)  -- (0,-4);
\draw (0,4) -- (2.5,-1)  -- (0,-4);
\draw[dashed] (0,4) -- (0,1);
\draw[dashed]  (4,0.15) -- (0,1) -- (-4,0.15);

\draw[ultra thick] (0,4) -- (-2.5,-1) -- (0,-4);
\draw[ultra thick] (-4,0.15)  -- (0,-4);
\draw[ultra thick] (4,0.15)  -- (0,-4);
\draw[ultra thick] (2.5,-1)  -- (0,-4);
\draw[ultra thick, dashed] (0,1) -- (0,-4);

\end{tikzpicture}
\caption{One triangulation for $n=7$ with a spanning tree in bolded lines}
\label{fig:casen7}
\end{center}
\end{figure}
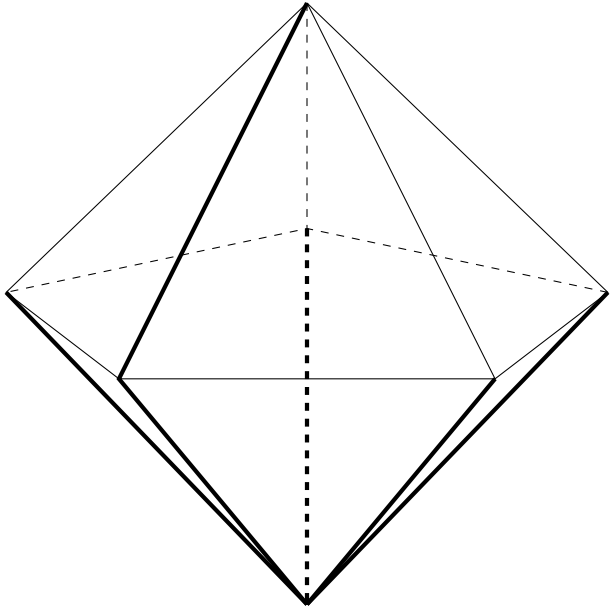
We may improve on this by constructing a non-arithmetic example as follows. First, we note that the arithmetic example given in Figure \ref{fig:casen7}, if we place $\infty$ at the top vertex,  can be seen as being generated by the matrices
\[ a_1 = \begin{pmatrix}1 & 5 \\ 0 & 1\end{pmatrix}, a_2 = \begin{pmatrix}1 & 0 \\ 4 & 1\end{pmatrix}, a_3 = \begin{pmatrix}5 & -4 \\ 4 & -3\end{pmatrix}, a_4 = \begin{pmatrix}9 & -16 \\ 4 & -7\end{pmatrix}, a_5 = \begin{pmatrix}13 & -36 \\ 4 & -11\end{pmatrix}, a_6 = \begin{pmatrix}17 & -64 \\ 4 & -15 \end{pmatrix}.     \]   
We may perturb this and take the group generated by
\[ b_1 = \begin{pmatrix}1 & 5 \\ 0 & 1\end{pmatrix}, b_2 = \begin{pmatrix}1 & 0 \\ 4 & 1\end{pmatrix}, b_3 = \begin{pmatrix}53/11 & -441/110 \\ 40/11 & -31/11 \end{pmatrix}, b_4 = \begin{pmatrix}47/5 & -441/25 \\ 4 & -37/5\end{pmatrix},\]
\[ b_5 = \begin{pmatrix}1331/97 & -380689/9700 \\ 400/97 & -1137/97\end{pmatrix}, b_6 = \begin{pmatrix}569/31 & -217083/3100 \\ 400/93 & -507/31 \end{pmatrix}.     \]   
This has the effect of moving and slightly changing the waist sizes of the cusps. The group is again cofinite, and generated by 6 parabolics so gives a surface with 7 cusps. The five words which previously all had trace 14 now have traces 14.0364, 14.0364, 14.0037, 14.0071, 14.0211 respectively, and none of the elements which previously had trace 18 or more have been reduced below 14. To check this latter assertion, we may use Lemma 2.3 of Benson--Lakeland--Then \cite{BLT}. Hence, this is one example of a non-arithmetic surface with 7 cusps which has a longer systole than any arithmetic surface of the same type. We note that this example does not maximize the systole, as we may perturb further, and so one may ask what length systole may be achieved in this manner, and whether another construction may produce an even longer systole in this case.\\
\subsection{$n=8$} 

\begin{prop}Every regular planar triangulation with $8$ vertices has an edge of density at most $18$. \end{prop}

\begin{proof}Consider the tetrahedron with four degree 3 vertices. Stellate each face; that is, introduce a new vertex inside each face, and join it to each vertex of that face with an edge. The result has eight vertices. Each original vertex now has degree 6, and each new vertex has degree 3, but no two new vertices are adjacent. Hence, the minimum edge density here is 18.  

In the presence of a degree 3 vertex, with $G$ regular, this is the best one can do. For example, if there were a degree 3 vertex only adjacent to three degree 7 vertices, then the remaining 4 vertices must all have degree 3. In order for there to be no $(3,3)$ edge, each degree 3 vertex must be adjacent to all three degree 7 vertices. But then $G$ contains a $K_{3,3}$ subgraph, and hence is not planar.

If $G$ is regular and the minimum degree is $\delta(G)=4$, then in view of equation (\ref{grapheqn}), there must be at least four vertices of degree 4. If more than four, then two such must be adjacent, creating an edge of density 16. If exactly four, then the other four vertices have degree 5. If each degree 4 vertex is only adjacent to degree 5 vertices, then each degree 5 vertex is also adjacent to all of the degree 4 vertices, creating a $K_{4,4}$ subgraph. Thus in this case there must be an edge of density 16.\end{proof}

If $\delta(G)\geq 3$ but $G$ is not regular, then $G$ has a loop. Since the vertex degrees sum to 36, the maximum degree of a vertex in this case is $36 - (3 \times 7) = 15$, and thus by Lemma \ref{lem:loops}, $\Gamma$ has a hyperbolic element of trace at most $7$.

If $\delta(G) = 2$, then each degree 2 vertex must be adjacent to vertices of degrees at least 10. Since the sum of degrees is 36, the remaining five vertex degrees sum to at most 14, so there must be a second degree 2 vertex. But this must also be adjacent to the two degree 10 vertices, and Lemma \ref{adjtriangles} applies and gives an element of trace 14. If the degree 2 vertex is adjacent to one vertex of degree at least 10, this degree must be at most $36-(2\times 7) = 22$, and there must be a loop, so by Lemma \ref{lem:loops}, $\Gamma$ has a hyperbolic element of trace at most 11.

If $\delta(G) =1$, to increase the systole, the degree 1 vertex must be adjacent to a vertex $v$ of degree at least 19. This degree must be at most 29. There must be a loop based at this vertex, so Lemma \ref{lem:loops} gives that $\Gamma$ has a hyperbolic element of trace at most 14.

With 8 cusps, the systole among arithmetic examples is at most $2\arccosh{(8)}$. \\
\subsection{$n=9$} 

\begin{prop}Every regular planar triangulation with $9$ vertices has an edge of density at most $20$. \end{prop}

\begin{proof}Suppose first that $G$ is regular with minimum degree $\delta(G)= 4$. There is a spherical triangulation with three vertices of degree 4 and six of degree 5, where no two degree 4 vertices are adjacent (see Figure ). If there were a degree 4 vertex adjacent to four degree 6 vertices, then by equation (\ref{grapheqn}), some of the remaining vertices would have to have degree 3, putting us in the next case.

Now suppose that $G$ is regular and there is a vertex of degree 3. If it is adjacent to a vertex of degree 6 or lower, then there is an edge of density less than 20. So, suppose it is adjacent only to vertices of degree 7 or 8. But then, by equation (\ref{grapheqn}), there must be at least two more vertices of degree 3. In order to not have an edge of density lower than 20, these must all be adjacent to the three vertices of degrees 7 or 8. But then there is a $K_{3,3}$ subgraph and the graph is not planar.\end{proof}

If $\delta(G) \geq 3$ but $G$ is not regular, then there is a loop. By equation (\ref{grapheqn}), the maximum degree of the vertex where the loop is based is 18, if there are eight other degree three vertices. But if there is such a loop, there is a hyperbolic element of trace at most 9 by Lemma \ref{lem:loops}.

If $\delta(G)=2$, then each degree 2 vertex must be adjacent to vertices of degrees at least 11. Since the sum of degrees is 42, the remaining six vertex degrees sum to at most 18, so either there is another degree 2 vertex, or there are six degree 3 vertices. In the former case, a second degree 2 vertex would need the degree 11 vertices as neighbors, putting us in the situation where Lemma \ref{adjtriangles} applies and gives an element of trace 14. In the latter case, the degree 3 vertices cannot all be adjacent only to the degree 11 vertices, so there must be an edge of lower density. If the vertex with degree 11 or more is unique, then its degree must be at most 24 by equation (\ref{grapheqn}), and so Lemma \ref{lem:loops} gives a hyperbolic element of trace at most 12.

If $\delta(G)=1$, then to give a systole corresponding to an element of trace larger than 18, any degree 1 vertex must be adjacent to a vertex of degree at least 21. This degree is at most 34 by equation (\ref{grapheqn}). Here there must be a loop, and Lemma \ref{lem:loops} gives a hyperbolic element of trace at most 17.

With 9 cusps, the systole among arithmetic examples is at most $2\arccosh{(9)}$. \\
\subsection{$n=10$} 

\begin{prop}Every regular planar triangulation with $10$ vertices has an edge of density at most $20$.\end{prop}

\begin{proof} Suppose first that $G$ is regular, and the minimum degree of the graph is $\delta(G)=4$. There is a graph with eight degree 5 vertices and two degree 4 vertices, where the two degree 4 vertices are non-adjacent -- see Example 2, and Figure \ref{fig:casen10}. Suppose a degree 4 vertex was adjacent to only degree 6 vertices. Then by equation (\ref{grapheqn}) the other five vertices must all have degree 4. But then the degree 6 vertices only have three outgoing edges to these, making a total of 12, but the five degree 4 vertices have 20 edges outgoing. Thus, two of the degree 4 vertices must be adjacent.

Now suppose that $G$ is regular and $\delta(G)=3$. To obtain a longer systole, suppose that every edge of $G$ has density greater than 20. Let $v_1$ be such a degree 3 vertex, and let $w_1, w_2, w_3$ denote its neighbors, whose degrees must be at least 7. Then deleting that vertex produces a 9-vertex spherical triangulation $G'$, which by the above case $n=9$ has an edge of density at most 20. This edge must be one whose density was reduced by the deletion, so it must have one vertex which was adjacent to the deleted vertex, i.e. one vertex is $w_1, w_2$ or $w_3$. This vertex has degree at least 6 after the deletion, and so the other vertex $v_2$ on the edge of density at most 20 must be degree 3. In $G$, $v_2$ must also be adjacent to only vertices of degree at least 7. If there are two (or more) additional vertices $w_4,w_5$ of $G$ of degree at least 7, then the degrees of $v_1,v_2,w_1,w_2,w_3,w_4$ and $w_5$ sum to at least 41, meaning that by equation (\ref{grapheqn}) there must be a vertex of degree at most 2, a contradiction. If there are none, and $v_2$ is also adjacent to $w_1, w_2$ and $w_3$, then deleting both $v_1$ and $v_2$ produces an 8-vertex triangulation $G''$, which by a previous case must have an edge of density at most 18. If this edge is not incident to $w_1, w_2$ or $w_3$ then it also has density 18 in $G$, a contradiction. So it is incident to one of these vertices, each of which has degree at least 5 in $G''$. Thus the other vertex $v_3$ has degree 3, and also has degree 3 in $G$. But then $v_3$ is adjacent to $w_1, w_2$ and $w_3$, and so $v_1, v_2$ and $v_3$ and $w_1, w_2$ and $w_3$ form a $K_{3,3}$. It then remains to consider the case that there are exactly four vertices of degree at least 7. Suppose, without loss of generality, that $v_2$ is adjacent to $w_1, w_2$ and $w_4$. Then by equation (\ref{grapheqn}), the remaining four vertices have degree which sum to at most 14. Therefore at least two of these degrees must be 3, and none may be greater than 5. Since $w_3$ is adjacent to $v_1$, $w_1$ and $w_2$, and not to $v_2$, $w_3$ must be adjacent to at least three of these vertices, and at least two must be adjacent to one another in the link of $w_3$. The edge incident to these two vertices has density at most 16, a contradiction.\end{proof}

If $\delta(G) \geq 3$ but $G$ is not regular, then $G$ must have a loop. By the constraints imposed by equation (\ref{grapheqn}) and the minimum degree being at least 3, the maximum degree of any vertex is 21. But then, by Lemma \ref{lem:loops}, there must be a hyperbolic element of trace at most 11.

If $\delta(G) \geq 2$ with no loops, suppose there is no edge of density $D \leq 20$. Then there exists a vertex $v_1$ of degree 2, which must be adjacent to two vertices $w_1, w_2$ of degrees at least 11. The remaining seven vertices then have degrees which sum to at most 24, so either there is a second degree 2 vertex, or there are at least four degree 3 vertices. If there is a second degree 2 vertex $v_2$, then either $v_2$ is adjacent to $w_1$ and $w_2$ -- in which case Lemma \ref{adjtriangles} gives an element of trace 14 -- or there must be a third vertex $w_3$ of degree at least 11. Then, the degrees of $w_1, w_2$ and $w_3$ sum to at least 33, meaning the other seven vertex degrees sum to at most 15. There are then six degree 2 vertices. At least two of these six must both be adjacent to either $\{ w_1, w_2 \}$, $\{ w_1, w_3 \}$, or $\{ w_2, w_3 \}$, and again Lemma \ref{adjtriangles} gives an element of trace 14. If there are no other degree 2 vertices, then there are at least four degree 3 vertices, $y_1, y_2, y_3, y_4$. Since there are no edges with density 20 or less, each of these must be adjacent only to vertices of degrees at least 7. Thus, there must be a third vertex $w_3$ of degree at least 7. If only three such, then $y_1, y_2, y_3$ and $w_1, w_2, w_3$ form a $K_{3,3}$. So ther emust also be a fourth such vertex $w_4$ with degree at least 7. Then, the degrees of the eight vertices, four $y_i$ and four $w_i$, sum to at least 48. This contradicts equation (\ref{grapheqn}).

If $\delta(G)=2$ with loops present, then the maximum degree of a vertex is 30, so Lemma \ref{lem:loops} gives a hyperbolic element of trace at most 15.

If $\delta(G)=1$, then the largest possible degree for a vertex is 39. If this were 37 or below, Lemma \ref{lem:loops} would give a hyperbolic element of trace at most 18, as required. It therefore remains to consider a triangulation $G$ with one vertex of degree 38 or 39. If there is a vertex of degree 39, then all other vertices have degree 1 and are arranged as a flower around the degree 39 vertex. In this case, then selecting a loop which encloses two degree 1 vertices, there must be a loop which creates a trace 7 element corresponding to a word of the form $L^5R$. If degree 38, the same applies, as there must be one degree 2 vertex, but among the 8 degree 1 vertices the same is true.

With $10$ cusps, the systole among arithmetic examples is at most $2\arccosh{(9)}$. \\

To obtain a longer systole with a non-arithmetic example, we adapt Example 2 above.

\begin{figure}[htb]
\begin{center}
\begin{tikzpicture}[scale=0.7]
\draw (0,-2) -- (0,2) -- (2,0) -- (-2,0) -- (0,-2);
\draw (-2,0) -- (0,2);
\draw (0,-2) -- (2,0);
\draw (3,3) -- (3,-3) -- (-3,-3) -- (-3,3) -- (3,3);
\draw (3,3) -- (2,0) -- (3,-3) -- (0,-2) -- (-3,-3) -- (-2,0) -- (-3,3) -- (0,2) -- (3,3);
\draw[ultra thick] (-3,3) -- (0,6) -- (3,3);

\draw[ultra thick] (0,6) .. controls (-8,3) and (-4,-2) .. (-3,-3);
\draw[ultra thick] (0,6) .. controls (8,3) and (4,-2) .. (3,-3);

\draw (0,0) -- (2,0) -- (0,2);
\draw (2,0) -- (0,-2) -- (-2,0);
\draw (3,3) -- (2,0) -- (3,-3);
\draw (-3,-3) -- (3,-3);
\draw (-3,3) -- (3,3) -- (0,6);

\draw (0.35,0.25) node {$v_\infty$};
\draw (1.5,0.2) node {$v_0$};
\draw (0.3,1.4) node {$v_1$};
\draw[ultra thick] (0,0) -- (2,0) -- (3,-3);
\draw[ultra thick] (0,2) -- (3,3);
\draw[ultra thick] (0,-2) -- (-3,-3);
\draw[ultra thick] (-2,0) -- (-3,3);

\end{tikzpicture}
\caption{Another spanning tree for Example 2}
\label{fig:long10}
\end{center}
\end{figure}
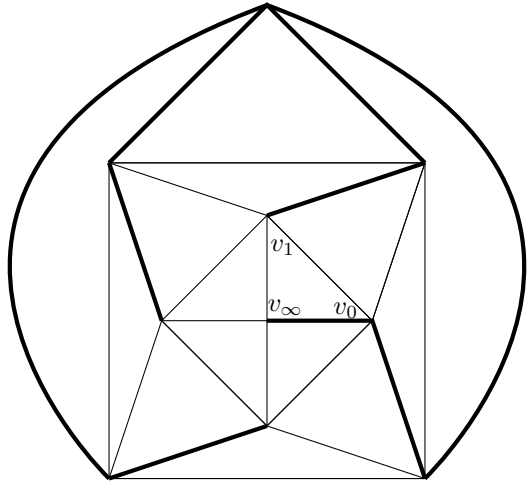

Developing the triangulation started in Figure \ref{fig:long10}, and with Lemma \ref{lem:fixed} in mind, we see that $\infty$ is fixed by $T^4$, and, respectively, the points $0,1,2,3,4, 1/2, 3/2, 5/2,$ and $7/2$ are fixed by conjugates of $T^5$. Using these as generators for the group $\Gamma$, we find that $\Gamma$ is generated by the following elements:

\[ \gamma_1 = \begin{pmatrix} 1 & 4 \\ 0 & 1\end{pmatrix}, \ \ \  \gamma_2 = \begin{pmatrix} 1 & 0 \\ 5 & 1\end{pmatrix}, \ \ \  \gamma_3 = \begin{pmatrix} 6 & -5 \\ 5 & -4\end{pmatrix}, \ \ \  \gamma_4 = \begin{pmatrix} 11 & -20 \\ 5 & -9\end{pmatrix}, \ \ \ \gamma_5 = \begin{pmatrix} 16 & -45 \\ 5 & 14\end{pmatrix} \]
\[ \gamma_6 = \begin{pmatrix} 11 & -5 \\ 20 & -9\end{pmatrix}, \ \ \  \gamma_7 = \begin{pmatrix} 31 & -45 \\ 20 & -29\end{pmatrix}, \ \ \  \gamma_8 = \begin{pmatrix} 51 & -125 \\ 20 & -49\end{pmatrix}, \ \ \  \gamma_9 = \begin{pmatrix} 71 & -245 \\ 20 & -69\end{pmatrix} \]

The following figure gives a Ford fundamental domain for this group:


With these generators, the eight words which correspond to the systoles, and which have trace 18, are: 

\[ \gamma_2 \gamma_1^{-1}, \gamma_3 \gamma_1^{-1}, \gamma_4 \gamma_1^{-1}, \gamma_5 \gamma_1^{-1}, \]
\[  \gamma_2 \gamma_1^{-1} \gamma_9 \gamma_5 \gamma_8 \gamma_4 \gamma_7 \gamma_3,  \gamma_7 \gamma_3 \gamma_2 \gamma_1^{-1} \gamma_9 \gamma_5 \gamma_8 \gamma_4 , \gamma_8 \gamma_4 \gamma_7 \gamma_3 \gamma_2 \gamma_1^{-1} \gamma_9 \gamma_5 , \gamma_9 \gamma_5 \gamma_8 \gamma_4 \gamma_7 \gamma_3 \gamma_2 \gamma_1^{-1}  \]

We perturb this example to a group generated by the following: we set 
\[ P = \begin{pmatrix} 1 & 101/100 \\ 0 & 1\end{pmatrix} \]

and our generators are

\[ \alpha_1 = \begin{pmatrix} 1 & 404/100 \\ 0 & 1\end{pmatrix}, \ \ \  \alpha_2 = \begin{pmatrix} 1 & 0 \\ 499/100 & 1\end{pmatrix}, \ \ \  \alpha_3 = P \alpha_2 P^{-1}, \ \ \  \alpha_4 =P^2 \alpha_2 P^{-2}, \ \ \ \alpha_5 = P^3 \alpha_2 P^{-3} \]
\[ \alpha_6 = \begin{pmatrix} 111197/10399 & -5090299/1039900 \\ 199600/10399 & -90399/10399\end{pmatrix}, \ \ \  \alpha_7 = P \alpha_6 P^{-1}, \ \ \  \alpha_8 = P^2 \alpha_6 P^{-2}, \ \ \  \alpha_9 =P^3 \alpha_6 P^{-3} \]

We note that the corresponding words all have traces $-18.1596$. We check using Lemma 2.3 of \cite{BLT} that there are no hyperbolic elements with trace 18 or less, and thus that this example does have longer systole.

\subsection{$n=11$} 

\begin{prop}Every regular planar triangulation with $11$ vertices has an edge of density at most $20$.\end{prop}

\begin{proof}Suppose there is no such edge. 

First, suppose that $\delta(G) = 4$. There must exist a vertex $v_1$ of degree 4, and it must be adjacent to four vertices $w_1, w_2, w_3, w_4$, each of degree at least 6. By equation (\ref{grapheqn}), the remaining six vertices have degrees which sum to at most 26. So, at least four of these vertices, $v_2, v_3, v_4, v_5$ must have degree 4. If we suppose that there are five vertices of degrees at least 6, then the only possibility which satisfies equation (\ref{grapheqn}) is that there are six vertices of degree 4 and five vertices of degree 6. There are 27 edges of $G$. If each edge has density larger than 20, 24 of these edges must be incident to one of the $v_i$ and one of the $w_i$. But $v_1$ is incident to four of the $w_i$, and the four edges which form the link of $v_1$ are all edges incident to two of the $w_i$. So there are at least 4 edges incident to none of the $v_i$. This is a contradiction. 

If there are only four vertices of degrees at least 6, then either there is a vertex of degree 5, or there are seven vertices of degree 4, and at least one vertex of degree larger than 6. Suppose first there is a vertex $z_1$ of degree 5. Then even if $z_1$ is adjacent to $w_1, w_2, w_3$ and $w_4$, $z_1$ must be adjacent to a fifth vertex $z_2$, which must also have degree 5. If the degree 4 vertices $v_1$ through $v_5$ are not adjacent to one another, then there are 20 edges emanating from them which must have a second vertex one of the $w_i$. There are 24 edges emanating from the $w_i$, as all four must have degree 6. But 8 of these edges have $z_1$ or $z_2$ as an endpoint, leaving at most 16 to have one of the $v_i$ as an endpoint, a contradiction. Now suppose there is no vertex of degree 5. Then there are seven vertices of degree 4, which have 28 edges incident to them. There are four other vertices, with degrees summing to 26. Therefore two of the degree 4 vertices must be adjacent to one another, a contradiction.

Now suppose that $\delta(G) = 3$ and $G$ is regular. Then there is a vertex $v_1$ with degree 3, adjacent to three vertices $w_1, w_2$ and $w_3$, each of degrees at least 7. Consider the triangulation $G'$ obtained by deleting $v_1$. We see that $G'$ has 10 vertices, and so by the previous case, it must have an edge $e$ with density at most 20. At least one end point of $e$ must be one of $w_1, w_2$ or $w_3$, as otherwise $e$ would have density at most 20 in $G$. Since the degrees of  $w_1, w_2$ and $w_3$ in $G'$ are at least 6, the other endpoint of $e$ must be a vertex $v_2$ with degree 3. In $G$, $v_2$ may only be adjacent to vertices of degrees at least 7. We study cases based on how many such vertices there are.

If there are six or more vertices of $G$ with degrees at least 7, then their degrees sum to at least 42. By equation (\ref{grapheqn}), all eleven vertex degrees sum to 54, but since $\delta(G) =3$, we find a contradiction.

If there are 3 vertices of degrees at least 7, then deleting both $v_1$ and $v_2$ produces a triangulation $G''$ with 9 vertices. By a previous case, $G''$ has an edge of density at most 20. This edge must be one whose density was reduced by the deletion, so it must have one vertex which was adjacent to the deleted vertex, i.e. one vertex is $w_1, w_2$ or $w_3$, each of which has degree at least 5 in $G''$. Let $y_3$ denote the other vertex of this edge; $y_3$ must have degree 3 or 4. If degree 3, then $v_1, v_2, y_3$ and $w_1, w_2, w_3$ form a $K_{3,3}$. If $y_3$ has degree 4, then $y_3$ is not adjacent to all of $w_1, w_2$ and $w_3$ (as this would create a $K_{3,3}$), but $y_3$ must be adjacent to four distinct vertices, each of which has degree at least 6. Since at most two of these are among the $w_i$, there are two more vertices $z_1, z_2$, which by assumption must both have degree 6. The degrees of the vertices $v_1, v_2, y_3, w_1, w_2, w_3, z_1$ and $z_2$ sum to at least 43. Equation (\ref{grapheqn}) then gives that one of the remaining three vertices, $v_3$, must have degree 3. But then $v_3$ must be adjacent to $w_1, w_2$ and $w_3$, forming a $K_{3,3}$.

If there are 4 vertices of degree at least 7, suppose $v_2$ is adjacent to $w_1$, $w_2$, and $w_4$. The remaining five vertices have degree which sum to at most 20, so at most two may have degree 5 or larger, and so at least three have degree 3 or 4. If one, $v_3$, has degree 3, then $v_3$ must be adjacent to three of the four $w_i$, and in particular to $w_3$, $w_4$ and one of $w_1$ or $w_2$; without loss of generality, suppose $v_3$ is adjacent to $w_1$. But then $w_1$ is adjacent to $v_1$, $w_2$, $v_2$, $w_4$, $v_3$ and $w_3$, and each of these is adjacent to its predecessor and successor in the list, including $w_3$ adjacent to $v_1$. So, if $w_1$ is adjacent to a seventh vertex, that vertex would be adjacent to one of the $v_i$, contradicting the assumption that these have degree 3. But then $w_1$ cannot have degree 7.

If there are 5 vertices, $w_1$ through $w_5$, of degree at least 7, then by equation (\ref{grapheqn}), the other six vertex degrees sum to at most 19, so either there are six degree 3 vertices, or five of degree 3 and one of degree 4. Each vertex of degree 7 must be adjacent to at most three degree 3 vertices, and any vertex of degree 8 to at most four, else two degree 3 vertices would be adjacent to one another. Delete five degree 3 vertices. The resulting graph $H$ is a 6-vertex triangulation where the $w_i$ all have degree at least 4 and there is one remaining low-degree vertex from $G$ which has degree 3 or 4. By equation (\ref{grapheqn}), the degrees of the vertices are either all 4, or there are four vertices with degree 4 and one each of degrees 3 and 5. If all 4's, one checks that the octahedron has no choice of five faces which if stellated produce a graph like $G$ was assumed to be. If degree 3 and 5, then deleting the degree 3 vertex creates a 5-vertex triangulation $H'$ with no vertices of degree less than 3; there is only one such and it is the triangulation discussed above in the case $n=5$. But stellating any five of the six faces of this triangulation cannot produce a triangulation with the properties $G$ was assumed to have, because one of the degree 3 vertices of $H'$ must only have two of its three incident faces stellated, meaning that the resulting triangulation will have an edge of density $3\times5 = 15$. \end{proof}

If $\delta(G) =2$ with no loops, then we have a vertex $v_1$ of degree 2, which must be adjacent to two distinct vertices $w_1$ and $w_2$ of degrees at least 11. By equation (\ref{grapheqn}), the remaining eight vertices have degrees which sum to at most 30, so either there is another degree 2 vertex, or there are at least two degree 3 vertices. If there is another vertex $v_2$ of degree 2, then either Lemma \ref{adjtriangles} gives an element of trace 14, or there is a third vertex $w_3$ of degree at least 11. There cannot be a fourth, since then the minimum sum of vertex degrees is 58, a contradiction. The three $w_i$ have degrees which sum to at least 33, meaning that the other eight vertex degrees sum to at most 21. If there are only three degree 2 vertices, then there are five of degree 3; each of these must be adjacent to only $w_1$, $w_2$, and $w_3$, meaning there must be a $K_{3,3}$. If there are more than three degree 2 vertices, two must be adjacent to the same pair of the $w_i$, and them Lemma \ref{adjtriangles} gives a trace 14 element. Suppose now there is only one degree 2 vertex. By equation (\ref{grapheqn}), there are at least two degree 3 vertices, $y_1$ and $y_2$. Each must be adjacent only to vertices of degree at least 7. Thus, there must be a third vertex $z_1$ of degree at least 7. Suppose that there is a third vertex $y_3$ of degree 3. Then if there is no fourth vertex of degree at least 7, each vertex of degree 3 must be adjacent to each of those three, creating a $K_{3,3}$. So, suppose there is a fourth vertex $z_2$ with degree at least 7. Then the degrees of $w_1, w_2, z_1$ and $z_2$ sum to at least 36, so the remaining seven vertex degrees sum to at most 18. But this implies that there is a second vertex of degree 2, a contradiction. Finally, suppose there is no fourth vertex of degree at least 7, and therefore no third vertex of degree 3, so all other vertices aside from $v_1$, $z_1$ and $z_2$ have degree at least 4. But the degrees of $w_1, w_2, z_1, v_1, y_1$ and $y_2$ sum to at least 37, so this contradicts equation (\ref{grapheqn}). 

If $\delta(G) \geq 2$ and there are loops, then the maximal degree of a vertex is 34, so Lemma \ref{lem:loops} gives an element of trace at most 17.

If $\delta(G) = 1$, then the maximum degree of a vertex in $G$ is 44. If the maximum degree is 37 or below, Lemma \ref{lem:loops} gives a hyperbolic element of trace at most 18. If 38 or above, then there are at least 4 degree 1 vertices, and either two enclosed by the same triangle -- giving a trace 7 element as in the proof in the case $n=10$ -- or one has a triangle where the adjacent triangle has a degree 2 or degree 3 vertex, and Lemma \ref{lem:deg1s} gives a trace 6 or trace 10 hyperbolic element. 

With $11$ cusps, the systole among arithmetic examples is at most $2\arccosh{(9)}$. \\

In Figure \ref{fig:long11}, an example of an 11-vertex triangulation which achieves this systole is given, via six edges of density 20.

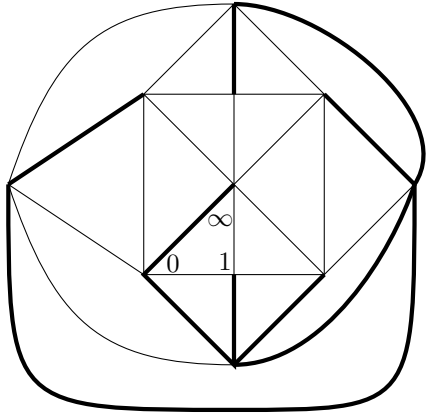
\begin{figure}[htb]
\begin{center}
\begin{tikzpicture}[scale=0.6]
\draw (-2,-2) -- (-2,2) -- (2,2) -- (2,-2) -- (-2,-2);
\draw (0,-2) -- (0,2);
\draw (-2,-2) -- (2,2);
\draw (-2,2) -- (2,-2);
\draw (2,2) -- (4,0) -- (2,-2);
\draw (-2,-2) -- (0,-4) -- (2,-2);
\draw (0,-2) -- (0,-4);
\draw (2,2) -- (0,4) -- (-2,2);
\draw (0,2) -- (0,4);
\draw (-2,-2) -- (-5,0) -- (-2,2); 

\draw[ultra thick] (0,0) -- (-2,-2) -- (0,-4) -- (0,-2);
\draw[ultra thick] (2,-2) -- (0,-4);
\draw[ultra thick] (0,4) -- (0,2);
\draw[ultra thick] (4,0) -- (2,2);
\draw[ultra thick] (-5,0) -- (-2,2);

\draw (-5,0) .. controls (-4,3) and (-3,4) .. (0,4);
\draw (-5,0) .. controls (-4,-3) and (-3,-4) .. (0,-4);
\draw[ultra thick] (0,-4) .. controls (2,-4) and (3.5,-1.5) .. (4,0);
\draw[ultra thick] (0,4) .. controls (2,4) and (5,1.5) .. (4,0);
\draw[ultra thick] (-5,0) .. controls (-5,-5) .. (0,-5);
\draw[ultra thick] (4,0) .. controls (4,-5) .. (0,-5);

%
%
%
\draw (-0.3,-0.8) node {$\infty$};
\draw (-1.35,-1.75) node {$0$};
\draw (-0.2,-1.7) node {$1$};

\end{tikzpicture}
\caption{There are six edges of density 20}
\label{fig:long11}
\end{center}
\end{figure}

Though Figure \ref{fig:long11} provides a fundamental domain, we will instead use the (more complicated) Ford fundamental domain obtained from the same generators. This is shown in Figure \ref{fig:ford11}, with generators listed below.

\begin{figure}
\begin{center}

\includegraphics[scale=0.3]{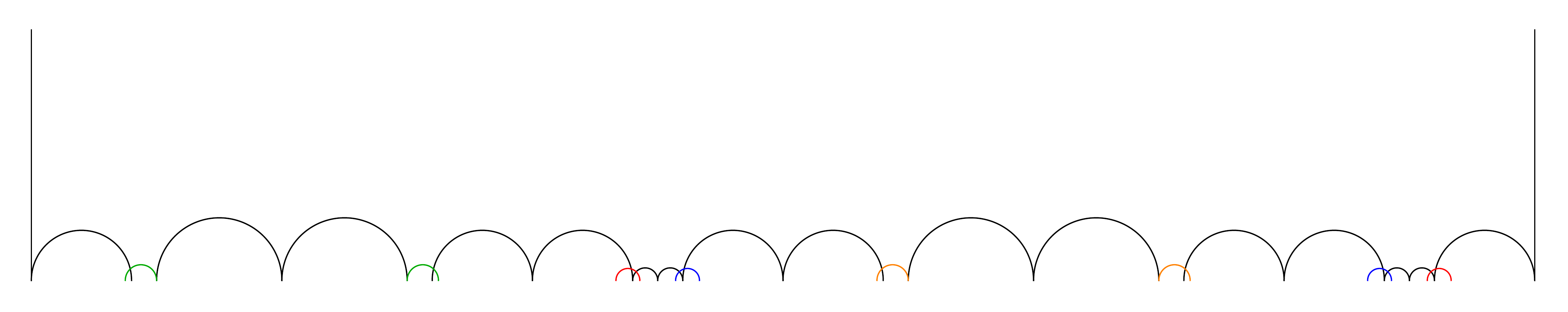}

\caption{Ford domain for the 11-cusp arithmetic group $\Gamma$}
\label{fig:ford11}
\end{center}
\end{figure}

\[ \gamma_1 = \begin{pmatrix}1 & 6 \\ 0 & 1\end{pmatrix} \ \ \ \  \gamma_2 = \begin{pmatrix}-29 & 6 \\ -5 & 1\end{pmatrix} \ \ \ \  \gamma_3 = \begin{pmatrix}5 & -4 \\ 4 & -3\end{pmatrix} \ \ \ \  \gamma_4 = \begin{pmatrix}11 & -20 \\ 5 & -9\end{pmatrix} \]
\[  \gamma_5 = \begin{pmatrix}16 & -45 \\ 5 & -14\end{pmatrix} \ \ \ \  \gamma_6 = \begin{pmatrix}17 & -64 \\ 4 & -15\end{pmatrix} \ \ \ \  \gamma_7 = \begin{pmatrix}26 & -125 \\ 5 & -24\end{pmatrix} \]
\[  \gamma_8 = \begin{pmatrix}25 & -11 \\ 16 & -7\end{pmatrix} \ \ \ \  \gamma_9 = \begin{pmatrix}-73 & 251 \\ -16 & 55\end{pmatrix} \ \ \ \  \gamma_{10} = \begin{pmatrix}-49 & 125 \\ -20 & 51\end{pmatrix} \]
\[  \gamma_{11} = \begin{pmatrix}111 & -605 \\ 20 & -109\end{pmatrix} \ \ \ \  \gamma_{12} = \begin{pmatrix}-118 & 281 \\ -21 & 50\end{pmatrix} \ \ \ \  \gamma_{13} = \begin{pmatrix}-113 & 296 \\ -21 & 55\end{pmatrix} \]

With these generators, the six words with trace 18 are as follows:
\[ \gamma_8, \gamma_9, \gamma_4 \gamma_3, \gamma_6 \gamma_5, \gamma_7\gamma_6, \gamma_3 \gamma_2^{-1}\gamma_1 .\]

To find a non-arithmetic example with longer systole, and setting
\[ \tau = \begin{pmatrix}1 & 1/100 \\ 0 & 1\end{pmatrix}, \]

we use the following generators:

\[ \alpha_1 = \begin{pmatrix}1 & 602/100 \\ 0 & 1\end{pmatrix} \ \ \ \  \alpha_2 = \tau^2 \gamma_2 \ \ \ \  \alpha_3 = \begin{pmatrix}503/101 & -40401/10100 \\ 400/101 & -301/101\end{pmatrix} \ \ \ \  \alpha_4 = \tau \gamma_4 \tau^{-1} \]
\[  \alpha_5 = \tau \gamma_5 \tau^{-1} \ \ \ \  \alpha_6 = \begin{pmatrix}1707/101 & -644809/10100 \\ 400/101 & -1505/101\end{pmatrix} \ \ \ \  \alpha_7 = \tau^2 \gamma_7 \tau^{-2} \]
\[  \alpha_8 = \tau \gamma_8 \ \ \  \alpha_9 = \tau^2 \gamma_9 \tau^{-1} \ \ \ \  \alpha_{10} = \tau \gamma_{10} \tau^{-1} \]
\[  \alpha_{11} = \tau^2 \gamma_{11} \tau^{-2} \ \ \ \  \alpha_{12} = \tau^2 \gamma_{12} \tau^{-1} \ \ \ \  \alpha_{13} = \tau^2 \gamma_{13} \tau^{-1} \]

With respect to these, the corresponding words have traces $\pm 454/25 = \pm 18.16$ or $\pm 36361/2020 \approx \pm 18.000495$. We again use Lemma 2.3 of \cite{BLT} to show that there are no shorter loops, so the systole here corresponds to the trace $\pm 36361/2020$ elements, and is longer than that of the arithmetic examples with 11 cusps.

\subsection{$n=12$}  The principal congruence subgroup $\Gamma(5)$ has 12 cusps. The graph obtained as the quotient of the Farey graph by $\Gamma(5)$ is an icosahedron with all edge densities 25; therefore, this surface has 30 systoles, all of length $2\arccosh{(23/2)}$. When $n=12$, the bound given by equation (\ref{bound}) is $4\arccosh{(5/2)}$, and we see that this is equal to $2\arccosh{(23/2)}$. Hence, when $n=12$, Schmutz's bound is realized by $\Gamma(5)$ and the systole length is maximized by this arithmetic group.

\section{Further Questions}

Among the cases considered above, the most obvious question is how long the systole can be when the maximum systole length has not yet been determined. 

Beyond the cases considered here, i.e. for $n>12$ cusps, it remains to be shown for which values of $n$ the systole length is maximized by an arithmetic surface. We suspect that there is some $N \in \mathbb{N}$ such that for all $n>N$, the systole length among hyperbolic punctured spheres with $n$ cusps is maximized by a non-arithmetic surface. This is because it follows from a result of Borodin \cite{Borodin} that every regular planar triangulation $G$ with $\delta(G) \geq 3$ has an edge of density at most 30, but as $n \to \infty$ the systole bound given by Schmutz increases towards a limit which corresponds to edges of density no less than 36. If, furthermore, every non-regular planar triangulation has an edge of density no more than 30, then there should be room to increase the systole length beyond that corresponding to a hyperbolic element of trace 28. It may be that a starting point for this investigation will be those triangulations $G$ whose dual graphs are fullerene graphs, i.e. trivalent planar graphs with each face a pentagon or hexagon. In such triangulations $G$, the vertices of $G$ all have degree 5 or degree 6, and if no two degree 5 vertices are adjacent (i.e. no two pentagons in the dual are adjacent) then the minimal edge density in $G$ will be 30.

\bibliographystyle{amsplain}
\bibliography{sphererefs}

\providecommand{\bysame}{\leavevmode\hbox to3em{\hrulefill}\thinspace}
\providecommand{\MR}{\relax\ifhmode\unskip\space\fi MR }
\providecommand{\MRhref}[2]{%
  \href{http://www.ams.org/mathscinet-getitem?mr=#1}{#2}
}
\providecommand{\href}[2]{#2}
\begin{thebibliography}{10}

\bibitem{BLT}
Brian~A. Benson, Grant~S. Lakeland, and Holger Then, \emph{Cheeger constants of
  hyperbolic reflection groups and {M}aass cusp forms of small eigenvalues},
  Proc. Amer. Math. Soc. \textbf{149} (2021), no.~1, 417--438. \MR{4172617}

\bibitem{Borodin}
O.~V. Borodin, \emph{Joint generalization of the theorems of {L}ebesgue and
  {K}otzig on the combinatorics of planar maps}, Diskret. Mat. \textbf{3}
  (1991), no.~4, 24--27. \MR{1160234}

\bibitem{BrooksMakover}
Robert Brooks and Eran Makover, \emph{Random construction of {R}iemann
  surfaces}, J. Differential Geom. \textbf{68} (2004), no.~1, 121--157.
  \MR{2152911}

\bibitem{FanoniParlier}
Federica Fanoni and Hugo Parlier, \emph{Systoles and kissing numbers of finite
  area hyperbolic surfaces}, Algebr. Geom. Topol. \textbf{15} (2015), no.~6,
  3409--3433. \MR{3450766}

\bibitem{FBR}
Maxime Fortier~Bourque and Kasra Rafi, \emph{Local maxima of the systole
  function}, J. Eur. Math. Soc. (JEMS) \textbf{24} (2022), no.~2, 623--668.
  \MR{4382480}

\bibitem{ReidBook}
Colin Maclachlan and Alan~W. Reid, \emph{The arithmetic of hyperbolic
  3-manifolds}, Graduate Texts in Mathematics, vol. 219, Springer-Verlag, New
  York, 2003. \MR{1937957}

\bibitem{McKay}
Brendan McKay, \emph{Product of vertex degrees of an edge in a planar graph},
  MathOverflow, URL:https://mathoverflow.net/q/407110 (version: 2021-10-26).

\bibitem{NeumannReid}
Walter~D. Neumann and Alan~W. Reid, \emph{Arithmetic of hyperbolic manifolds},
  Topology '90 ({C}olumbus, {OH}, 1990), Ohio State Univ. Math. Res. Inst.
  Publ., vol.~1, de Gruyter, Berlin, 1992, pp.~273--310. \MR{1184416}

\bibitem{Schmutz}
Paul Schmutz, \emph{Congruence subgroups and maximal {R}iemann surfaces}, J.
  Geom. Anal. \textbf{4} (1994), no.~2, 207--218. \MR{1277506}

\bibitem{Series}
Caroline Series, \emph{The modular surface and continued fractions}, J. London
  Math. Soc. (2) \textbf{31} (1985), no.~1, 69--80. \MR{810563}

\bibitem{Takeuchi}
K.~Takeuchi, \emph{A characterization of arithmetic {F}uchsian groups}, J.
  Math. Soc. Japan \textbf{27} (1975), no.~4, 600--612. \MR{0398991 (53
  \#2842)}

\end{thebibliography}

\noindent Department of Mathematics \& Computer Science,\\
Eastern Illinois University,\\
600 Lincoln Avenue,\\
Charleston, IL 61920.\\
Email: gslakeland@eiu.edu\\

\noindent Department of Mathematics \& Computer Science,\\
Eastern Illinois University,\\
600 Lincoln Avenue,\\
Charleston, IL 61920.\\
Email: cyoung2@eiu.edu

\end{document}